\documentclass{amsart}

\usepackage[english]{babel}
\usepackage[T1]{fontenc}
\usepackage{amsmath,amssymb}
\usepackage[all]{xy}
\usepackage{graphicx}
\usepackage{fancyhdr}

\newtheorem{definition}{Definition }

\newtheorem{theorem}[definition]
{Theorem }

\newtheorem{corollary}[definition]
{Corollary}
\newtheorem{prop}[definition]{Proposition}

\newcommand{\lgw}{\longrightarrow}

\newcommand{\ovl}{\overline}

\newcommand{\ord}{\text{ord}}

\newcommand{\wdh}{\widehat}
\newcommand{\ini}{\text{in}}

\renewcommand{\l}{\lambda}

\renewcommand{\O}{\mathcal{O}}
\newcommand{\mfk}{\mathfrak}
\newcommand{\m}{\mathfrak{m}}

\renewcommand{\k}{\Bbbk}

\newcommand{\N}{\mathbb{N}}

\renewcommand{\a}{\alpha}
\renewcommand{\b}{\beta}

\renewcommand{\phi}{\varphi}

\newcommand{\e}{\varepsilon}
\begin{document}
\title{Asymptotic behaviour  of standard bases}
\author{Guillaume ROND}
\address{Institut de Math\'ematiques de Luminy\\
Universit\'e Aix-Marseille 2\\
Campus de Luminy, case 907\\
13288 Marseille cedex 9}
\email{rond@iml.univ-mrs.fr}

\begin{abstract}We prove here that the elements of any standard basis of $I^n$, where $I$ is an ideal of a Noetherian local ring and $n$ is a positive integer, have order bounded by a linear function in  $n$. We deduce from this that the elements of any standard basis of $I^n$ in the sense of Grauert-Hironaka, where $I$ is an ideal of the ring of power series, have order bounded by a polynomial function in $n$.\end{abstract}
\maketitle
The aim of this note is to study the growth of the orders of the elements of a standard basis of $I^n$, where $I$ is an ideal of a Noetherian local ring. We show here that the maximal order of an element of a standard basis of $I^n$ is bounded by a linear function in $n$. For this we prove a linear version of the strong Artin-Rees lemma for ideals in a Noetherian ring. The main result of this note is theorem \ref{main}.\\
First we prove the following proposition inspired by corollary 3.3 of \cite{P-V}:\\
\begin{prop}\label{princ}
Let $A$ be a Noetherian ring and let $I$  and $J$ be ideals of $A$. There exists an integer $\l\geq 0$ such that
$$\forall x\in A, \forall n,m\in\N,\, n\geq \l m,\ \ (x^n)\cap (J+I^m)=((x^{\l m})\cap (J+I^m))(x^{n-\l m}).$$
\end{prop}
\begin{proof}
 Let $B:=A/J$. By theorem 3.4 of \cite{Sw}, there exists $\l$ such that, for any $m\geq 1$, there exists an irredundant primary decomposition $I^m=Q_1^{(m)}\cap\cdots\cap Q_r^{(m)}$ such that, if $P_i^{(m)}:=\sqrt{Q_i^{(m)}}$, then $(P_i^{(m)})^{\l m}\subset Q_i^{(m)}$ for $1\leq i\leq m$. We denote by $\ovl{Q}_i^{(m)}$ the image of $Q_i^{(m)}$ in $A/(J+I^m)$ for $1\leq i\leq r$. We denote by $\mathfrak{P}_i^{(m)}$ the inverse image of $P_i^{(m)}$ in $A$, for $1\leq i\leq r$.\\
Let $x\in A$. If $x\in \mfk{P}_i^{(m)}$, then $x^{n}\in (\mfk{P}_i^{(m)})^n$  and $(\ovl{Q_i}^{(m)}:x^n)=A/(J+I^m)$ for any $n\geq \l m$. If $x\notin \mfk{P}_i^{(m)}$, then $x^n\notin (\mfk{P}_i^{(m)})^n$ and $(\ovl{Q}_i^{(m)}:x^n)=\ovl{Q}_i^{(m)}$ for any $n\geq \l m$.
Thus, for any $n\geq \l m$:
$$\left(0_{A/(J+I^m)}:x^n\right)=\left(\bigcap_i\ovl{Q}_i^{(m)}:x^n\right)=\bigcap_i\left(\ovl{Q}_i^{(m)}:x^n\right)=\bigcap_{i\ /\ x\notin P_i}\ovl{Q}_i^{(m)}.$$
Hence, by remark 2 (1) of \cite{P-V} and theorem 2 of \cite{P-V}, we get the result.
 \end{proof}
Using the extended Rees algebra of $\mfk{a}$ to reduce to the principal case (as done in \cite{Sw}) we prove the following corollary:
\begin{corollary}\label{A-R}
Let $A$ be a Noetherian ring and let $I$, $J$ and $\mfk{a}$ be ideals of $A$. Then there exists $\l\geq 0$ such that
$$(J+I^m)\cap \mfk{a}^n=((J+I^m)\cap \mfk{a}^{\l m}) \mfk{a}^{n-\l m}$$
\end{corollary}

\begin{proof}
Let $B:=A[\mfk{a} t,t^{-1}]$. Then $t^{-n}B\cap A=\mfk{a}^n$. By proposition \ref{princ}, there exists $\l\geq 1$ such that for any $n,m\in\N$, $n\geq \l m$ , $(t^{-n})\cap(J+ I^m)=((t^{-\l m})\cap (J+I^m))(t^{-(n-\l m)})$.
We have $$(J+I^m)\cap \mfk{a}^n=((t^{-m})\cap (J+I^m)B)\cap A=\left(((t^{-\l m})\cap (J+I^m))(t^{-(n-\l m)})\right)\cap A.$$
Thus $(J+I^m)\cap \mfk{a}^n\subset ((J+I^m)\cap \mfk{a}^{\l m})\mfk{a}^{n-\l m}$. The reverse inclusion is clear.
\end{proof}
Let $(A,\m)$ be a Noetherian local ring and $I$ be an ideal of $A$. Let us denote by $G(A/I)$ the associated graded ring of $A/I$ with respect to $\m$. Then $G(A/I)=G(A)/I^*$ where $I^*\subset G(A)$ is the graded ideal of $G(A)$ generated by the elements $f^*$ with $f\in I$, where $f^*$ is the leading form of $f$: if $\ord(f):=\sup\{k,\ f\in \m^k\}=d$, then $f^*=f+\m^{d+1}$. Finally $f_1,...,f_p$ form a \textit{(minimal) standard basis} of $I$ if $f_1^*,...,f_p^*$ form a (minimal) generating set of $I^*$. It is clear that  $(I^*)^n$ is included in $(I^n)^*$, but both ideals are not equal in general. For example, if $I=(x^2,y^3-xy)\subset \k[[x,y]]$ where $\k$ is a field, then $(I^n)^*=((xy,x^2)^n,\{x^iy^{4n-3i+1}\}_{0\leq i\leq n-1})$, hence $y^{4n+1}\in (I^n)^*\backslash(I^*)^n$ \cite{C-H-S}.  Nevertheless we have the following theorem, whose proof is inspired by the link made in \cite{B-M} between the Artin-Rees lemma and the orders of the elements of a standard basis, with respect to a monomial order, of an ideal in the ring of formal power series over a field (see also \cite{Wa}):

\begin{theorem}\label{main}
Let $I$ be an ideal of a Noetherian local ring $(A,\m)$. Then there exists an integer $\l\geq 0$ such that for any integer $n\geq 0$ and any minimal standard basis $f_1,...,f_{p_n}$ of $ I^n$ we have $\ord(f_i)\leq \l n$ for $1\leq i\leq p_n$.
\end{theorem}

\begin{proof}
 The canonical morphism $A\lgw \wdh{A}$ is injective and $G(A/I)=G(\wdh{A/I})$. Thus we may assume that $A$ is complete. Then $A$ is of the form $B/J$ where $B$ is a regular local ring and $J$ is an ideal of $B$. Hence we may assume that $A$ is a regular local ring, $I$ and $J$ are ideals of $A$, and we need to prove that there exists $\l\geq 0$ such that for any minimal standard basis $f_1,...,f_{p_n}$ of $J+ I^n$ we have $\ord(f_i)\leq \l n$ for $1\leq i\leq p_n$.\\
Let us assume that $I+J\neq (0)$. Let $n\in\N^*$ and let $f_1,...,f_{p_n}\in J+I^n$  such that $f_1^*,...,f_{p_n}^*$ form a minimal generating set of $(J+I^n)^*$ (in particular $(f_1,...,f_{p_n})=J+I^n$). Let us denote by $r_i$ the integer $\ord(f_i)$, $1\leq i\leq p_n$ and let us assume that $r_1\leq r_2\leq \cdots\leq r_{p_n}$. Let $\l\geq 0$ satisfying corollary \ref{A-R} with $\mfk{a}=\m$. Let $q\geq 0$ such that $r_i\leq \l n$ for $i\leq q$ and $r_i>\l n$ for $i>q$. It is enough to show that $q=p_n$. Let us assume that $q<p_n$. If $q=0$, then $f_i\in (J+I^n)\cap \m^{r_{i}}=((J+I^n)\cap \m^{\l n})\m^{r_{i}-\l n}\subset (J+I^n)\m$, $1\leq i\leq p_n$. Hence $(J+I^n)=\m(J+ I^n)$, and $(J+I^n)=(0)$ by Nakayama, which is a contradiction. Thus $q\geq 1$. For $i>q$ we have $f_{i}\in (J+I^n)\cap \m^{r_{i}}=((J+I^n)\cap \m^{\l n})\m^{r_{i}-\l n}$. Thus, for $q+1\leq i\leq p_n$,
  $f_i=\sum_k \e_{i,k}g_{i,k}$ with $g_{i,k}\in (J+I^n)\cap \m^{\l n}$ $\e_{i,k}\in\m^{r_i-\l n}$, for $q<i\leq p_n$ and any $k$.  Hence
  $f_i=\sum_k \e_{i,k}\left(\sum_{1\leq l\leq p_n}\eta_{i,k,l}f_l\right)$
  with  $\eta_{i,k,l}\in\m^{\l n-r_l}$, for any $i,k,l$ (because $f_1^*,...,f_{p_n}^*$ generate $(J+I^n)^*$ and $G(A)$ is an integral domain). Thus, for $q<i\leq p_n$,
  $$f_{i}=(1-\sum_k\e_{i,k}\eta_{i,k,i})^{-1}\sum_k\e_{i,k}\left(\sum_{l\neq i}\eta_{i,k,l}f_l\right).$$
  Then 
  $f_i\in\sum_{l\neq i}f_l\m^{r_{i}-r_l}$
  for $q+1\leq i\leq p_n$. By Gaussian elimination we see that 
    $$f_i\in\sum_{l< i}f_l\m^{r_{i}-r_l}\text{  for }q+1\leq i\leq p_n.$$
    This means that $f_i^*\in(f_1^*,...,f_{i-1}^*)G(A)$ which contradicts the fact that $f_1^*,...,f_{p_n}^*$ form a minimal generating set of $(J+I^n)^*$.  
  \end{proof}
 Let $\O_s:=\k[[x_1,...,x_s]]$ where $\k$ is a field or $\O_s:=\k\{x_1,....,x_s\}$ where $\k$ is a valued field. We denote by $\m$ its maximal ideal. For all $\a\in\N^s$ let us denote  $|\a|:=\a_1+\cdots+\a_s$.  We define a total order on $\N^s$ in the following way: $\a> \b$ if $(|\a|,\a_1,...,\a_s)>_{lex}(|\b|,\b_1,...,\b_s)$ for all $\a,\b\in\N^s$. This induces a total order on the monomials of $\O_s$ in the following way: $x^{\a}>x^{\b}$ if $\a>\b$ for all $\a,\b\in\N^s$. If $f=\sum_{\a\in\N^s}f_{\a}x^{\a}\in \O_s$ let us denote by $\ini_>(f)$ the element $f_{\a}x^{\a}$ such that $\a<\b$ for all $\b\neq \a$ such that $f_{\b}\neq 0$. If $\ini_>(f)=f_{\a}x^{\a}$ let us denote by $\exp(f)$ the element $\a\in\N^s$. Let $I$ be an ideal of $\O_s$; we say that  $(f_1,...,f_p)$ is a \textit{(minimal) standard basis of $I$ with respect to this order} if $\{\exp(f_1),...,\exp(f_p)\}$ is a (minimal) set of generators of the semigroup $\{\exp(g), g\in I\}$ (in particular  $(f_1,...,f_p)=I$). We denote $\a_i:=\exp(f_i)$ for all $i$. We may always assume that $|\a_1|\leq \cdots\leq|\a_p|$. In this case, for $l\in \N$ we  define $q(l)\in\N$ by $\a_{q(l)}\leq l$ and $\a_{q(l)+1}>l$ where $q(l)=0$ if $l<|\a_1|$ and $q(l)=p$ if $l\geq |\a_p|$. 
We have the following result:
\begin{prop}\label{b-m}\cite{Wa}
Let $I$ be and ideal of $\O_s$. Then, with the previous notation,
$$I\cap\m^{m+l}=(I\cap \m^l)\m^m \text{ for all }m\geq 0$$ if and only if $r(l)\geq 1$ and  $f_j\in \m^{|\a_j|-|\a_1|}f_1+\cdots+\m^{|\a_j|-|\a_{r(l)}|}f_{r(l)}$, for $j=r(l)+1,...,p$.
 \end{prop}

\begin{corollary}\label{case2}
Let $I$ be an ideal of $\O_s$. Then there exists a polynomial function in $n$, denoted by $P$, such that for all integer $n\geq 0$ and any minimal standard basis $f_1,...,f_{p_n}$ of $ I^n$ with respect to $\ini_>$ we have $\ord(f_i)\leq P(n)$ for $1\leq i\leq p_n$.
\end{corollary}

\begin{proof}
Let $f_1,..., f_{p_n}$ be a minimal standard basis of $I^n$ with respect to $\ini_>$. Let $\a_i:=\exp(f_i)$, $1\leq i\leq p_n$ and let us assume that $\a_1\leq \a_2\leq \cdots\leq \a_{p_n}$. The sequence, $\a_1,...,\,\a_{p_n}$ is uniquely determined by $I^n$. By applying proposition \ref{b-m} and corollary \ref{A-R}, we see that there exists $\l\geq 0$, not depending on $n$, such that $|\a_1|\leq |\a_2|\leq\cdots\leq |\a_r|\leq \l n<|\a_{r+1}|\leq \cdots\leq |\a_{p_n}|$
$$\text{ and } f_i\in \m^{|\a_i|-|\a_1|}f_1+\cdots+\m^{|\a_i|-|\a_r|}f_{r}\ \ \ \text{ for } r+1\leq i\leq p_n.$$
In particular $(f_1^*,...,f_r^*)$ is a system of generators of $(I^n)^*$, and $(f_1^*,....,f_{p_n}^*)$ is a Gr\"obner basis of the homogeneous ideal $(I^n)^*$ with respect to the graded lexicographic order. From \cite{M-M}, $\ord(f_i^*)$ is bounded by a polynomial function in $\l n$ depending only on $I$ and $s$, for $r+1\leq i\leq p_n$. This proves the corollary.

\end{proof}

The author would like to thank  Irena Swanson for having taken time to answer his many questions about the subject.

\bibliographystyle{amsplain}

\begin{thebibliography}{10}
\bibitem{B-M} E. Bierstone, P. Milman, The local geometry of analytic mappings, Universita di Pisa, ETS Editrice Pisa, 1988.
\bibitem{C-H-S} S. D. Cutkosky, J. Herzog, H. Srinivasan, Finite Generation of Algebras Associated to Powers of Ideals, Arxiv, preprint.
\bibitem{M-M} H. M\"oller, F. Mora,  Upper and lower bounds for the degree of Groebner bases,  \textit{EUROSAM 84 (Cambridge, 1984)},  172-183, \textit{Lecture Notes in Comput. Sci.}, \textbf{174}, Springer, Berlin, 1984.
\bibitem{P-V} F. Planas Vilanova, The strong uniform Artin-Rees property in codimension one, \textit{J. reine angew. Math.}, \textbf{527}, (2000), 185-201.
 \bibitem{Sw} I. Swanson, Primary decomposition, Artin-Rees lemma an regularity, \textit{Math. Ann}, \textbf{307}, (1997), 3071-3079.
\bibitem{Wa} T. Wang, A stratification given by Artin-Rees estimates, \textit{Can. J. Math}, \textbf{44} (1), (1992), 194-205.
\end{thebibliography}

\end{document}